\documentclass[twoside, 11pt]{article}
\usepackage[top=1in, bottom=1in, left=1.5in, right=1.5in]{geometry}
\usepackage{enumerate,float}
\usepackage{section, amsthm, textcase, setspace, amssymb, lineno, amsmath, amssymb, amsfonts, latexsym, fancyhdr, longtable, ulem}
\usepackage{section, amsthm, textcase, setspace, amssymb, lineno, amsmath, amssymb, amsfonts, latexsym, fancyhdr, longtable, ulem}
\usepackage{tikz}
\usetikzlibrary{decorations.markings,decorations.pathreplacing,patterns}
\tikzstyle{vertex}=[circle, draw, inner sep=0pt, minimum size=13pt]
\newcommand{\vertex}{\node[vertex]}

\newtheorem{thm}{Theorem}[section]

\newtheorem{cor}[thm]{Corollary}

\theoremstyle{definition}

\newtheorem{example}[thm]{Example}

\newcommand{\mf}{\mathfrak}
\newcommand{\wh}{\widehat}
\newcommand{\ul}{\underline}
\newcommand{\Cn}{C_{\leq n}}
\newcommand{\dd}{\hspace{.1cm}|\hspace{.1cm}}
\newcommand{\ind}{{\rm ind \hspace{.1cm}}}
\newcommand{\rk}{{\rm rk \hspace{.1cm}}}
\newcommand{\A}{{\rm A}}
\newcommand{\C}{{\rm C}}

\newcommand{\lf}{\left\lfloor}
\newcommand{\rf}{\right\rfloor}

\newcommand{\Z}{\mathbb Z}

\begin{document}

\title{\bf Symplectic meanders\footnote{The results of this paper were delivered by the second author in a talk entitled \textit{Symplectic Meanders} in a January 2015 conference at the University of Miami in honor of his adviser Michelle Wachs.  Some of these results, inclusive of the meander construction, have been independently obtained by D. Panyushev and O. Yakimova per a recent arXiv post on January 3, 2016 \textbf{\cite{Panyushev2}}.}}
\author{Vincent E. Coll, Jr., Matthew Hyatt, and Colton Magnant}

\maketitle

\noindent
\small
\textit{Department of Mathematics, Lehigh University, Bethlehem, PA, USA}\\
\textit{Mathematics Department, Pace University, Pleasantville, NY, USA}\\
\textit{Department of Mathematical Sciences, Georgia Southern University, Statesboro, GA, USA
}

\begin{abstract}
\noindent
Analogous to the $\mathfrak{sl}(n)$ case, we address the computation of the index of seaweed subalgebras of $\mathfrak{sp}(2n)$ by introducing graphical representations called symplectic meanders. Formulas for the algebra's index may be computed by counting the connected components of its associated meander.  In certain cases, formulas for the index can be given in terms of elementary functions.
\end{abstract}



\section{Introduction}
The \textit{index} of a Lie algebra $\mf{g}$ is an important invariant and may be regarded as a generalization of the Lie algebra's rank:  $\ind \mf{g}\le \rk  \mf{g}$,  with equality when $\mf{g}$ is reductive.  More formally, the index of a Lie algebra $\mf{g}$ is given by

\[\ind \mf{g}=\min_{f\in \mf{g^*}} \dim  (\ker (B_f))\]

\noindent where $B_f$ is the associated skew-symmetric \textit{Kirillov form} defined by $B_f(x,y)=f([x,y])$ for all $x,y\in\mf{g}$.  Of particular interest are those Lie algebras for which the index is equal to zero.  Such algebras are called \textit{Frobenius} and have been studied extensively from the point of view of invariant theory \textbf{\cite{Ooms}} and are of special interest in deformation and quantum group theory stemming from their connection with the classcial Yang-Baxter equation (see \textbf{\cite{G1}} and \textbf{\cite{G2}}).

\textit{Seaweed algebras}, along with their very descriptive name, were first introduced by Dergachev and A. Kirillov in \textbf{\cite{Kirillov}}, where they defined such algebras as subalgebras of
$\mathfrak{gl}(n)$ preserving certain flags of subspaces developed from two compositions (ordered partitions) of $n$ (see Section 2.1). Subsequently, Panyushev \textbf{\cite{Panyushev1}} extended the (Lie theoretic) definition to reductive algebras: If $\mf{p}$ and $\mf{p'}$ are parabolic subalgebras of a reductive Lie algebra $\mf{g}$ such that $\mf{p}+\mf{p'}=\mf{g}$, then $\mf{p}\cap\mf{p'}$ is called a \textit{seaweed subalgebra o}f $\mf{g}$ or simply $seaweed$ when $\mathfrak{g}$ is understood.  Elsewhere, Joseph \textbf{\cite{Joseph}} has called seaweed algebras, \textit{biparabolic}.

To facilitate the computation of the index of seaweed subalgebras of $\mathfrak{sl}(n)$, the authors in \textbf{\cite{Kirillov}} also introduced the notion of a \textit{meander} --  a planar graph representation of the seaweed algebra.  The index of the seaweed can then be computed based on the number of connected components of the meander.  Thus producing, when the compositions have a small number of parts, explicit formulas for the index in terms of elementary functions whose arguments are the terms in the compositions of $n$. In particular, seaweeds are Frobenius when their associated meander graph consists of a single tree.  The latter result amplifies a now classical result of Elashvili \textbf{\cite{Elash}} which asserts that a maximal parabolic subalgebra of $\mathfrak{sl}(n)$, say $\mathfrak{p}((a,b)\dd (n))$, has index  $(a,b)-1$ and so is Frobenius precisely when $a$ and $b$ are relatively prime. In \textbf{\cite{Coll1}}, Coll et al.~produced additional explicit formulas, establishing, in particular, that the seaweed $\mathfrak{p}((a,b,c)\dd (n))$ is Frobenius when $(a+b,b+c)=1$; and in \textbf{\cite{Collar,Coll2}}, they introduced the notion of a meander's \textit{signature}, which renders Panyushev's well-known reduction into a deterministic sequence of graph theoretic moves.

The signature provides a fast algorithm (linear time in the number of vertices) for the computation of the index of a Lie algebra associated with the meander, allows for the speedy determination of the graph's plane \textit{homotopy type} (a finer invariant than the index), and can be used to construct arbitrarily large sets of meanders, Frobenius and otherwise, of any given size and configuration. More importantly, the signature can be used to test any relatively prime conditions which might serve to identify a Frobenius seaweed subalgebra of $\mathfrak{sl}(n)$. Indeed, using signature moves and complexity arguments, Karnauhova and Liebscher \textbf{\cite{Kar}} show, in particular, that there is no linear gcd formula for finding the index of the general seaweed $\mathfrak{p}((a_1,\dots, a_k)\dd(n))$, where $k\geq 4$.  This establishes that the formulas in \textbf{\cite{Coll2}} are, in some sense, the only ``nice'' ones.

Here, we advance this entire line of inquiry by examining seaweed subalgebras of $\mathfrak{sp}(2n)$ and after the fashion of the graphical approach detailed above introduce the notion of a \textit{symplectic meander}, which may be associated to a seaweed subalgebra of $\mathfrak{sp}(2n)$.  We establish relatively prime conditions for a symplectic seaweed to be Frobenius and find that the associated meander must reduce to a certain type of forest. As before, we provide relatively prime conditions for the seaweed to be Frobenius in all reasonable cases.

We assume throughout that the group field is algebraically closed and of characteristic zero, although much of what we do remains true in finite characteristic.  We also take the index, homotopy type etc.~of a meander to mean the index, homotopy type etc.~of its associated seaweed.

\section{Type A - $\mathfrak{sl}(n)$}

\subsection{Seaweeds}
Let $\mf{p}$ and $\mf{p'}$ be two parabolic subalgebras of a simple Lie algebra $\mf{g}$.  If $\mf{p} + \mf{p'}= \mf{g}$ then
$\mf{p}\cap\mf{p'}$ is called a seaweed subalgebra of $\mf{g}$. We assume that $\mf{g}$ is equipped with a triangular decomposition

$$
\mf{g}=\mf{u_+}\oplus\mf{h}\oplus\mf{u_-}
$$

\noindent
where $\mf{h}$ is a Cartan subalgebra of $\mf{g}$ and $\mf{u_+}$ and $\mf{u_-}$ are the subalgebras consisting of the upper and lower triangular matrices, respectively. Let $\Pi$ be the set of $\mf{g}$'s simple roots and for $\alpha\in\Pi$,
let $\mf{g}_{\alpha}$ denote the root space corresponding to $\alpha$. A seaweed subalgebra $\mf{p}\cap\mf{p'}$ is called \textit{standard} if
$\mf{p}\supseteq \mf{h}\oplus\mf{u}_+$ and $\mf{p'}\supseteq \mf{h}\oplus\mf{u_-}$.
In the case that $\mf{p}\cap\mf{p'}$ is standard, let
$\Psi=\{\alpha\in\Pi :\mf{g}_{-\alpha}\notin \mf{p}\}$,
$\Psi'=\{\alpha\in\Pi :\mf{g}_{\alpha}\notin \mf{p'}\}$,
and denote the seaweed by $\mf{p}(\Psi \dd \Psi')$.

Let $\mf sl(n)$ be the algebra of $n\times n$ matrices with trace zero and consider the triangular decomposition of $\mf{sl}(n)$ as above.  Let $\Pi=\{\alpha_1,\dots ,\alpha_{n-1}\}$ be the set of simple roots of $\mf sl(n)$ with the standard ordering and let let $\mf{p}_n^\A(\Psi \dd \Psi')$ denote a seaweed subalgebra of $\mf sl(n)$ where $\Psi$ and $\Psi'$ are subsets of $\Pi$.

Let $C_n$ denote the set of strings of positive integers whose sum is $n$
(i.e., $C_n$ is the set of compositions of $n$).
It will be convenient to index seaweeds of $\mf{sl}(n)$ by pairs of elements
of $C_n$. Let $\mathcal{P}(X)$ denote the power set of a set $X$.
Let $\varphi_\A$ be the usual bijection from $C_n$ to a set of cardinality $n-1$.
That is, given $\ul{a}=(a_1,a_2,\dots ,a_m)\in C_n$, define
$\varphi_\A :C_n\rightarrow \mathcal{P}(\Pi)$ by
\[\varphi_\A(\ul{a})=\{\alpha_{a_1},\alpha_{a_1+a_2},\dots
,\alpha_{a_1+a_2+\dots +a_{m-1}}\}.\]

\noindent
Then define $$\mf{p}_n^\A(\ul{a} \dd \ul{b})
=\mf{p}_n^\A(\varphi_\A(\ul{a}) \dd \varphi_\A(\ul{b})).$$

\noindent

By construction, the sequence of numbers in $\ul{a}$ determine the heights of triangles
below the main diagonal in $\mf{p}_n^\A(\ul{a} \dd \ul{b})$ which may have nonzero entries,
and the sequence of numbers in $\ul{b}$ determine the heights of triangles
above the main diagonal.
For example, the seaweed $\mf{p}_7^\A((4,3) \dd (2,2,1,2))
=\mf{p}_7^\A(\{\alpha_4\} \dd \{\alpha_2,\alpha_4,\alpha_5\})$
has the following shape, where * indicates the possible nonzero entries. See
the left side of Figure \ref{Aseaweed} below.
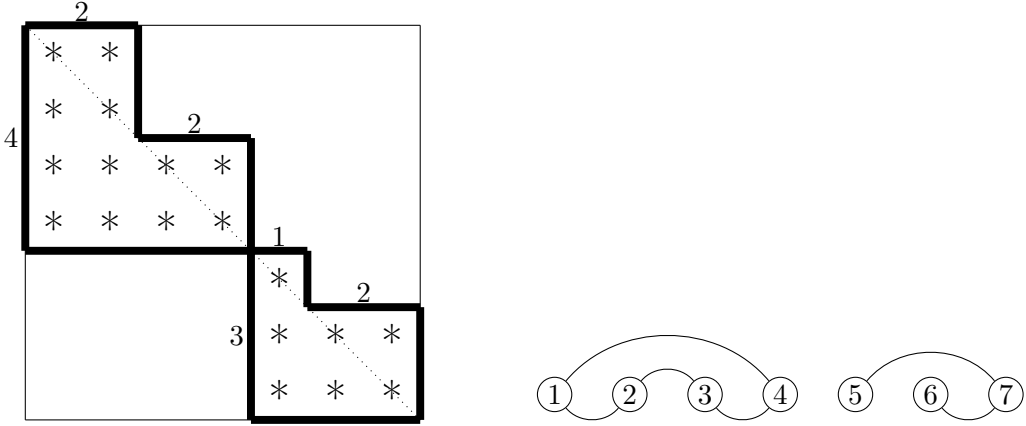
\begin{figure}[H]
\[\begin{tikzpicture}[scale=0.75]
\draw (0,0) -- (0,7);
\draw (0,7) -- (7,7);
\draw (7,7) -- (7,0);
\draw (7,0) -- (0,0);
\draw [line width=3](0,7) -- (0,3);
\draw [line width=3](0,3) -- (4,3);
\draw [line width=3](4,3) -- (4,0);
\draw [line width=3](4,0) -- (7,0);

\draw [line width=3](0,7) -- (2,7);
\draw [line width=3](2,7) -- (2,5);
\draw [line width=3](2,5) -- (4,5);
\draw [line width=3](4,5) -- (4,3);
\draw [line width=3](4,3) -- (5,3);
\draw [line width=3](5,3) -- (5,2);
\draw [line width=3](5,2) -- (7,2);
\draw [line width=3](7,2) -- (7,0);

\draw [dotted] (0,7) -- (7,0);

\node at (.5,6.4) {{\LARGE *}};
\node at (1.5,6.4) {{\LARGE *}};
\node at (.5,5.4) {{\LARGE *}};
\node at (1.5,5.4) {{\LARGE *}};
\node at (.5,4.4) {{\LARGE *}};
\node at (1.5,4.4) {{\LARGE *}};
\node at (2.5,4.4) {{\LARGE *}};
\node at (3.5,4.4) {{\LARGE *}};
\node at (.5,3.4) {{\LARGE *}};
\node at (1.5,3.4) {{\LARGE *}};
\node at (2.5,3.4) {{\LARGE *}};
\node at (3.5,3.4) {{\LARGE *}};
\node at (4.5,2.4) {{\LARGE *}};
\node at (4.5,1.4) {{\LARGE *}};
\node at (5.5,1.4) {{\LARGE *}};
\node at (6.5,1.4) {{\LARGE *}};
\node at (4.5,0.4) {{\LARGE *}};
\node at (5.5,0.4) {{\LARGE *}};
\node at (6.5,0.4) {{\LARGE *}};

\node at (-.25,5) {4};
\node at (3.75,1.5) {3};
\node at (1,7.25) {2};
\node at (3,5.25) {2};
\node at (4.5,3.25) {1};
\node at (6,2.25) {2};

\end{tikzpicture}
\hspace{1.5cm}
\begin{tikzpicture}[scale=1]
\vertex (1) at (1,0) {1};
\vertex (2) at (2,0) {2};
\vertex (3) at (3,0) {3};
\vertex (4) at (4,0) {4};
\vertex (5) at (5,0) {5};
\vertex (6) at (6,0) {6};
\vertex (7) at (7,0) {7};

\path
(1) edge[bend left=50] (4)
(2) edge[bend left=50] (3)
(5) edge[bend left=50] (7)
(1) edge[bend right=50] (2)
(3) edge[bend right=50] (4)
(6) edge[bend right=50] (7)
;\end{tikzpicture}
\]
\caption{
$\mf{p}_7^\A((4,3) \dd (2,2,1,2))$ and its associated
meander}
\label{Aseaweed}
\end{figure}


\subsection{Meanders}

Given a seaweed $\mf{p}^\A(\ul{a} \dd \ul{b})$ in $\mf{sl}(n)$,  Dergachev and A. Kirillov \textbf{\cite{Kirillov}} showed
how to associate to each such seaweed a planar graph (we use the word graph to mean a loopless, 2-edge-colored, multigraph)
called a \textit{meander}, which we denote $M^\A(\ul{a} \dd \ul{b})$.
We label the vertices of $M^\A(\ul{a} \dd \ul{b})$ as $1, 2, \dots ,n$ from left to right along a
horizontal line. We place edges above the horizontal line, called top edges, according to $\ul{a}$ as follows.
Let $\ul{a}=(a_1,a_2,\dots , a_m)$, and let $V_i$ be the $i^{\text{th}}$ block of vertices, that is the subset
of vertices whose label is greater than $a_1+a_2+\dots +a_{i-1}$ and less than $a_1+a_2+\dots +a_i+1$.
For each block $V_i$, place top edges connecting vertex $j$ to vertex $k$ if
$j+k=2(a_1+a_2+\dots+a_{i-1})+a_i+1$. In the same way,
place bottom edges according to $\ul{b}$. See the right side of Figure 1.

Note that $a_i$ is odd if and only if the vertex at the center of $V_i$ is not incident with a top edge, and
a similar statement follows for the bottom. Each vertex is incident with at most one top edge, and at most one
bottom edge. Thus we define a top bijection $t$ on $[n]$ by $t(j)=k$ if there is a top edge from vertex $j$ to
vertex $k$, and $t(j)=j$ if vertex $j$ is not incident with a top edge. Similarly we define a bottom bijection $b$ on $[n]$
by $b(j)=k$ if there is a bottom edge from vertex $j$ to vertex $k$,
and $b(j)=j$ if vertex $j$ is not incident with a bottom edge. Given a meander $M^\A(\ul{a} \dd \ul{b})$,
let $\sigma_{\ul{a},\ul{b}}$ be its associated permutation defined by $\sigma_{\ul{a},\ul{b}}(j)=t(b(j))$.
For example if $\ul{a}=(4,3)$ and $\ul{b}=(2,2,1,2)$ then the associated permutation written as a product of
disjoint cycles is $\sigma_{\ul{a},\ul{b}}=(1,3)(2,4)(5,7,6)$.

The following result of Dergachev and A. Kirillov allows us to compute the index of a seaweed by counting the number of connected components of its associated meander.
(Note that the formulas below differ from those appearing in \textbf{\cite{Kirillov}} by
one, since we are working in $\mf{sl}(n)$ instead of $\mf{gl}(n)$.)

\begin{thm}[\textbf{\cite{Kirillov}}, Theorem 5.1]\label{DKformula}
\
\begin{enumerate}[(i)]
\item The index of a seaweed $\mf{p}^\A(\ul{a} \dd \ul{b})$ is equal to the number of connected
components plus the number of cycles in the graph $M^\A(\ul{a} \dd \ul{b})$, minus one.
\
\item The index of a seaweed $\mf{p}^\A(\ul{a} \dd \ul{b})$ is equal to the number of cycles in the
disjoint cycle decomposition of $\sigma_{\ul{a},\ul{b}}$, minus one.
\end{enumerate}

\end{thm}

\begin{example}
The index of $\mf{p}^\A((4,3) \dd (2,2,1,2))$ is two.
\end{example}

We also have the following necessary condition for a Type A meander to be Frobenius.  This follows immediately from Theorem {\ref {DKformula} (cf., Corollary 4.7, \textbf{\cite{Panyushev1}}).

\begin{cor}\label{2 odd parts}

If $\ind \mf{p}^\A(\ul{a} \dd \ul{b})$ is Frobenius, then there are exactly 2 odd integers among $\ul{a}$ and $\ul{b}$.

\end{cor}

\section{Signature}

The \textit{signature} of a meander is a sequence of graph-theoretic ``moves'' which  deterministically winds a meander down to a unique representative form called the meander's (plane) \textit{homotopy type} consisting of nested circles (cycles) and points (vertices). From there, Theorem 2.1 can be readily applied to compute the index. For example, the homotopy type of the meander in Figure 1 is a single circle and and point exterior to the circle. Note that the homotopy type of a meander is a finer invariant than the index.

There are five basic moves in the signature, four of which replace a given meander with a homotopically equivalent one. The remaining move (component elimination) changes the index of the meander by eliminating a set of connected components. See \textbf{\cite{Coll2}} for details and examples.

\begin{thm}[\textbf{\cite{Coll2}}, Lemma 4]\label{edge contraction}

Consider the meander $M^\A(\ul{a} \dd \ul{b})$ where
$\ul{a}=(a_1,a_2,\dots ,a_m)$ and $\ul{b}=(b_1,b_2,\dots ,b_t)$. A component elimination move (C) removes cycles (and possibly a single vertex)
among the vertices $1,2,\dots ,a_1$. A pure contraction move (P) contracts the bottom edges connecting the vertices $1,2,\dots ,a_1$, and deletes these vertices. A block elimination move (B) and a rotation contraction move (R) contract the bottom edges of the
vertices $1,2,\dots ,b_1-a_1$, and deletes these vertices. And a flip move (F) simply exchanges $\ul{a}$ for $\ul{b}$.  In particular, only the component elimination move changes the homotopy type of the meander.

\end{thm}

It follows from Theorem 2.1 that a seaweed is Frobenius precisely when its associated meander is homotopically trivial.  Note also that the ``Winding Down" moves of the signature can be reversed to create ``Winding Up" moves which can be used to build all meanders, Frobenius and otherwise.

\subsection{Formulas}
While Theorem 2.1 provides an elegant formalism for computing the index of a seaweed, significant computational complexity persists.  What is needed is a mechanism for determining the index of a seaweed directly from its defining compositions. The first result of this kind is due to Elashvili.

\begin{thm}[\textbf{\cite{Elash}}, 1990] The maximal parabolic $\mf{p}^\A((a,b) \dd (n))$ has index $\rm{gcd}(a,b)-1.$
\end{thm}

\noindent
In (\textbf{\cite{Coll1}}, 2011), Coll et al. established that
$\mf{p}^\A((a,b,c) \dd (n))$ is Frobenius precisely when $\gcd(a+b,b+c)=1$ but what the index is in the non-Frobenius was left open.  The obvious guess for the index, $\gcd(a+b,b+c)-1$, is correct but the ad hoc techniques in that paper are insufficient to prove it.  Of serious note is that the theory was missing a transparent algorithmic method of generating more examples of Frobenius meanders - so that new formulas for seaweeds with more complicated compositions could be uncovered and tested.  This led to the signature which we observe is, in essence, a graph theoretic rendering of Panyushev's reduction. In \textbf{\cite{Coll2}}, and using the signature, Coll et al. established the following extension of Elashvili's theorem.

\begin{thm}[\textbf{\cite{Coll2}}, 2015]\label{A 4 parts}  The seaweeds $\mf{p}^\A((a,b,c) \dd (n))$
and $\mf{p}^\A((a,b) \dd (c,d))$
both have index $\gcd(a+b,b+c) -1$.
\end{thm}

One might conjecture the existence of similar ``closed" index formulas for more general seaweeds but using signature moves and complexity arguments Karnauhova and Liebscher have recently shown that there are severe restrictions.

\begin{thm}[\textbf{\cite{Kar}}, 2015]\label{5 parts} If $m\geq 4$ is given, then there do not exist homogeneous polynomials $f_1,f_2\in \Z[x_1,\dots ,x_m]$ of arbitrary degree such that the number of connected components of
$M^A_{n} ((a_1,\dots,a_m) \dd (n))$ is given by
$\gcd(f_1(a_1,\dots ,a_m),f_2(a_1,\dots ,a_m))$.

\end{thm}

\section{Type C - $\mathfrak{sp}(2n)$}

\subsection{Symplectic seaweeds}

Following the work on $\mf{sl}(n)$ we introduce symplectic seaweeds and associate to them planar graphs called symplectic meanders.  We show that  the index can be computed from simple graph theoretic properties of the symplectic meander, or from the number of
certain cycles in the
disjoint cycle decomposition of the associated permutation.

Let $\mf sp(2n)$ be the algebras of matrices with the following block form
\[\mf sp(2n)=\left\{\begin{bmatrix} A & B \\ C & -\wh{A}\end{bmatrix}: B=\wh{B}, C=\wh{C} \right\},\]
where $A, B,$ and $C$ are $n\times n$ matrices and $\wh{A}$ is the transpose of
$A$ with respect to the anitdiagonal.
Choose the same triangular decomposition as was done in the $\mf{sl}(n)$ case,
that is $\mf{sp}(2n)=\mf{u_+}\oplus\mf{h}\oplus\mf{u_-}$.
Let $\Pi=\{\alpha_1,\dots ,\alpha_{n}\}$ denote its set of simple roots,
where $\alpha_n$ is the exceptional root.
Let $\mf{p}_n^\C(\Psi \dd \Psi')$ denote a seaweed subalgebra where
$\Psi$ and $\Psi'$ are subsets of $\Pi$.

Let $C_{\leq n}$ denote the set of strings of positive integers whose sum
is less than or equal to $n$, and call each integer in the string a part.
It will be convenient for us to index seaweeds in $\mf{sp}(2n)$ by pairs of elements from $\Cn$.
Let $\mathcal{P}(X)$ denote the power set of a set $X$.
Given $\ul{a}=(a_1,a_2,\dots ,a_m)\in \Cn$, define a bijection $\varphi_\C:\Cn\rightarrow \mathcal{P}(\Pi)$ by
\[\varphi_\C(\ul{a})=\{\alpha_{a_1},\alpha_{a_1+a_2},\dots ,\alpha_{a_1+a_2+\dots +a_m}\}\]
Then define \[\mf{p}_n^\C(\ul{a} \dd \ul{b})=
\mf{p}_n^\C(\varphi_\C(\ul{a}) \dd \varphi_\C(\ul{b})).\]
Note that we need to keep the subscript $n$, since this is not determined by either $\ul{a}$ or $\ul{b}$.


For example let $n=3$ and consider the seaweed $\mf{p}_3^\C((2,1) \dd (1))
 =\mf{p}_3^\C(\{\alpha_2,\alpha_3\} \dd \{\alpha_1\})$. By construction,
this is the algebra of matrices in
$\mf{sp}(6)$ of the form in Figure \ref{Cseaweed} below,
where * indicates the possible nonzero entries.
\begin{figure}[H]
\[\begin{tikzpicture}
\draw (0,0) -- (0,6);
\draw (0,6) -- (6,6);
\draw (6,6) -- (6,0);
\draw (6,0) -- (0,0);
\draw [line width=3](0,6) -- (0,4);
\draw [line width=3](0,4) -- (2,4);
\draw [line width=3](2,4) -- (2,3);
\draw [line width=3](2,3) -- (3,3);
\draw [line width=3](3,3) -- (3,2);
\draw [line width=3](3,2) -- (4,2);
\draw [line width=3](4,2) -- (4,0);
\draw [line width=3](4,0) -- (6,0);
\draw [line width=3](0,6) -- (1,6);
\draw [line width=3](1,6) -- (1,5);
\draw [line width=3](1,5) -- (5,5);
\draw [line width=3](5,5) -- (5,1);
\draw [line width=3](5,1) -- (6,1);
\draw [line width=3](6,1) -- (6,0);
\draw [dotted] (0,3) -- (6,3);
\draw [dotted] (3,6) -- (3,0);
\draw [dotted] (0,6) -- (6,0);
\draw [dotted] (0,0) -- (6,6);

\node at (.5,5.4) {{\LARGE *}};
\node at (.5,4.4) {{\LARGE *}};
\node at (1.5,4.4) {{\LARGE *}};
\node at (2.5,4.4) {{\LARGE *}};
\node at (3.5,4.4) {{\LARGE *}};
\node at (4.5,4.4) {{\LARGE *}};
\node at (2.5,3.4) {{\LARGE *}};
\node at (3.5,3.4) {{\LARGE *}};
\node at (4.5,3.4) {{\LARGE *}};
\node at (3.5,2.4) {{\LARGE *}};
\node at (4.5,2.4) {{\LARGE *}};
\node at (4.5,1.4) {{\LARGE *}};
\node at (4.5,0.4) {{\LARGE *}};
\node at (5.5,0.4) {{\LARGE *}};
\node at (-.25,5) {2};
\node at (1.75,3.5) {1};
\node at (.5,6.25) {1};

\end{tikzpicture}\]
\caption{The shape of elements from $\mf{p}_3^\C((2,1) \dd (1))$}
\label{Cseaweed}
\end{figure}

The following theorems of Panyushev give inductive formulas for
computing the index of a symplectic seaweed.

\begin{thm}[\textbf{\cite{Panyushev1}}, Theorem 5.2]\label{C inductive}
Let $\ul{a}\neq\emptyset$ and $\ul{b}\neq\emptyset$.
Consider the seaweed $\mf{p}_n^\C(\ul{a} \dd \ul{b})$ where
$\ul{a}=(a_1,a_2,\dots ,a_m)$ and $\ul{b}=(b_1,b_2,\dots ,b_t)$.

\begin{enumerate}[(i)]
\item If $a_1=b_1$ then
\[\ind \mf{p}_n^\C(\ul{a} \dd \ul{b})
=a_1+\ind \mf{p}_{n-a_1}^\C((a_2,a_3,\dots a_m) \dd (b_2,b_3,\dots b_t)).\]

\

\item If $a_1<b_1$ then
\[\ind \mf{p}_n^\C(\ul{a} \dd \ul{b})=
\begin{cases}
\ind \mf{p}_{n-a_1}^\C((a_2,a_3,\dots a_m) \dd (b_1-2a_1,a_1,b_2,b_3,\dots b_t))
& \text{ if }a_1\leq b_1/2 \\
\ind \mf{p}_{n-b_1+a_1}^\C((2a_1-b_1,a_2,a_3,\dots a_m) \dd (a_1,b_2,b_3,\dots b_t))
& \text{ if }a_1> b_1/2.
\end{cases}\]

\end{enumerate}

\end{thm}

Note that if $a_1>b_1$, we can use the fact that
$\mf{p}_n^\C(\ul{a} \dd \ul{b})\cong\mf{p}_n^\C(\ul{b} \dd \ul{a})$.

\begin{thm}[\textbf{\cite{Panyushev1}}, Theorem 5.5]\label{C parabolic}
Let $\ul{a}=(a_1,a_2,\dots ,a_m)$.
For parabolic subalgebras in $\mf{sp}(2n)$ we have
\[\ind \mf{p}_n^\C(\ul{a} \dd \emptyset)=n-\left(\sum_{i=1}^{m}a_i\right)+\left(\sum_{i=1}^{m}\lf\frac{a_i}{2}\rf\right).\]
\end{thm}

Notice that for non-parabolic seaweeds in $\mf{sp}(2n)$, the inductive part of the formula is precisely
the same inductive formula as for seaweeds in $\mf{sl}(n)$. As a corollary to Panyushev's inductive formulas,
we notice that for computing
the index of $\mf{p}_n^\C(\ul{a} \dd \ul{b})$, it suffices to assume one of these strings sums to $n$.

\begin{cor}\label{sums<n} Consider the seaweed $\mf{p}_{n+k}^\C(\ul{a} \dd \ul{b})$ where
$\ul{a}=(a_1,a_2,\dots ,a_m)$ and $\ul{b}=(b_1,b_2,\dots ,b_t)$.
Suppose $\displaystyle n+k>n=\sum_{i=1}^{m}a_i\geq \sum_{i=1}^{t}b_i$, then
\[\ind \mf{p}_{n+k}^\C(\ul{a} \dd \ul{b})=k+\ind \mf{p}_n^\C(\ul{a} \dd \ul{b}).\]

\end{cor}

\subsection{Symplectic meanders}

Given a seaweed $\mf{p}_n^\C(\ul{a} \dd \ul{b})$, associate
to this seaweed a symplectic meander, which we denote $M_n^\C(\ul{a} \dd \ul{b})$.
The construction is very similar to (type A) meanders.
We label the vertices of $M_n^\C(\ul{a} \dd \ul{b})$ as $1, 2, \dots ,n$ from left to right along a
horizontal line. We begin by placing edges above the horizontal line, called top edges,
according to $\ul{a}$ as follows.
Let $\ul{a}=(a_1,a_2,\dots , a_m)$, and let $V_i$ be the $i^{\text{th}}$ block of vertices, that is the subset
of vertices whose label is greater than $a_1+a_2+\dots +a_{i-1}$ and less than $a_1+a_2+\dots +a_i+1$.
For each block $V_i$, place an edge from vertex $j$ to vertex $k$
if $j+k=2(a_1+a_2+\dots+a_{i-1})+a_i+1$. Next, in the same way,
place bottom edges according to $\ul{b}$.
Finally, we designate a special subset of the vertices $T=T_n(\ul{a}\dd\ul{b})$
called the\textit{ tail} of the symplectic meander as follows:
If $\ul{a}\in \Cn$, let $r=\sum a_i$ and define a subset of vertices
$T_n(\ul{a})=\{r+1,r+2,\dots ,n\}$ then $T_n(\ul{a}\dd\ul{b})$ is the symmetric difference of
$T_n(\ul{a})$ and $T_n(\ul{b})$, i.e.,
\[T=T_n(\ul{a}\dd\ul{b})=\left(T_n(\ul{a})\cup T_n(\ul{b})\right)\setminus\left(T_n(\ul{a})\cap T_n(\ul{b})\right).\]
Note that if $\sum b_i\leq \sum a_i$, then $T=T_n(\ul{a}\dd\ul{b})=T_n(\ul{b})\setminus T_n(\ul{a})$.

\bigskip
\noindent
\begin{example} The symplectic meander $M_{11}^\C((2,1,1,6) \dd (2,2,1,2))$ has tail $T=\{8,9,10\}$.
We color the tail vertices yellow and visualize the graph as follows.

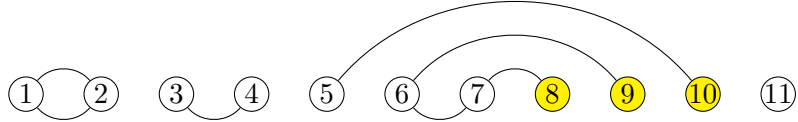
\begin{figure}[H]
\[\begin{tikzpicture}
\vertex (1) at (1,0) {1};
\vertex (2) at (2,0) {2};
\vertex (3) at (3,0) {3};
\vertex (4) at (4,0) {4};
\vertex (5) at (5,0) {5};
\vertex (6) at (6,0) {6};
\vertex (7) at (7,0) {7};
\vertex[fill=yellow] (8) at (8,0) {8};
\vertex[fill=yellow] (9) at (9,0) {9};
\vertex[fill=yellow] (10) at (10,0) {10};
\vertex (11) at (11,0) {11};


\path
(1) edge[bend left=50] (2)
(5) edge[bend left=50] (10)
(6) edge[bend left=50] (9)
(7) edge[bend left=50] (8)
(1) edge[bend right=50] (2)
(3) edge[bend right=50] (4)
(6) edge[bend right=50] (7)
;\end{tikzpicture}\]
\caption{The meander $M_{11}^\C((2,1,1,6) \dd (2,2,1,2))$}
\end{figure}

\noindent
Note that both strings $(2,1,1,6)$ and $(2,2,1,2)$ have sum less than 11.
We can easily obtain the graph $M_{10}^\C((2,1,1,6) \dd (2,2,1,2))$
from the graph $M_{11}^\C((2,1,1,6) \dd (2,2,1,2))$ by removing vertex 11,
and the tail remains the same.
\end{example}

Given a symplectic meander $M_n^\C(\ul{a} \dd \ul{b})$,
we define top and bottom bijections $t$ and $b$ exactly as we did before,
and and associate to the symplectic meanders a permutation $\sigma_{n,\ul{a},\ul{b}}$
defined by $\sigma_{n,\ul{a},\ul{b}}(j)=t(b(j))$. For example if $\ul{a}=(2,1,1,6)$ and $\ul{b}=(2,2,1,2)$
are string in $C_{\leq 11}$,
then the associated permutation written in disjoint cycle form
is $\sigma_{11,\ul{a},\ul{b}}=(1)(2)(3,4)(5,10)(6,8,7,9)(11)$.

The following theorem is the Type C analogue of the component formula of Dergachev and A. Kirillov (\ref{DKformula}) for the Type A case.

\begin{thm}\label{symplectic index}
Consider the seaweed $\mf{p}_n^\C(\ul{a} \dd \ul{b})$, and let $T=T_n(\ul{a}\dd\ul{b})$.

\begin{enumerate}[(i)]
\item The index of $\mf{p}_n^\C(\ul{a} \dd \ul{b})$ is equal to the number of cycles plus the number of connected
components containing either 0 or 2 vertices from $T$ in the graph $M_n^\C(\ul{a} \dd \ul{b})$.

\item  The index of $\mf{p}_n^\C(\ul{a} \dd \ul{b})$ is equal to the number of cycles containing either
0 or 2 integers from $T$ in the disjoint cycle decomposition of $\sigma_{n,\ul{a},\ul{b}}$
(here we view $T$ as a set of integers).

\end{enumerate}

\end{thm}

\begin{example}
With the running example, let $G=M_{11}^\C((2,1,1,6) \dd (2,2,1,2))$, and we have $T=\{8,9,10\}$.
We compute the contribution to the index
from each connected component as follows: If the component contains either 0 or 2 vertices from
the tail, then its contribution is exactly the same as in the type A case.  That is, a cycle contributes 2 to the
index, and a path or single vertex contributes 1 to the index. If a component contains exactly 1 vertex
from the tail, then it must be a path or single vertex, and it contributes nothing to the index.
Let $G[S]$ denote the subgraph induced by a set of vertices $S$. The table below
computes the contribution to the index for each component. It follows that
$\ind \mf{p}_{11}^\C((2,1,1,6) \dd (2,2,1,2))=5$.

\begin{figure}[H]
\[\begin{tikzpicture}
\vertex (1) at (1,0) {1};
\vertex (2) at (2,0) {2};
\vertex (3) at (3,0) {3};
\vertex (4) at (4,0) {4};
\vertex (5) at (5,0) {5};
\vertex (6) at (6,0) {6};
\vertex (7) at (7,0) {7};
\vertex[fill=yellow] (8) at (8,0) {8};
\vertex[fill=yellow] (9) at (9,0) {9};
\vertex[fill=yellow] (10) at (10,0) {10};
\vertex (11) at (11,0) {11};

\path
(1) edge[bend left=50] (2)
(5) edge[bend left=50] (10)
(6) edge[bend left=50] (9)
(7) edge[bend left=50] (8)
(1) edge[bend right=50] (2)
(3) edge[bend right=50] (4)
(6) edge[bend right=50] (7)
;\end{tikzpicture}\]

\[\def\arraystretch{1.2}
\begin{tabular}{c||c||c||c}
component & \# of tail vertices & cycle? & contribution to index \\
\hline
$G[\{1,2\}]$ & 0 & yes & 2 \\
$G[\{3,4\}]$ & 0 & no & 1 \\
$G[\{\{11\}]$ & 0 & no & 1 \\
$G[\{6,7,8,9\}]$ & 2 & no & 1 \\
$G[\{5,10\}]$ & 1 & no & 0 \\
\end{tabular}\]
\caption{$M_{11}^\C((2,1,1,6) \dd (2,2,1,2))$ with index computed}
\end{figure}
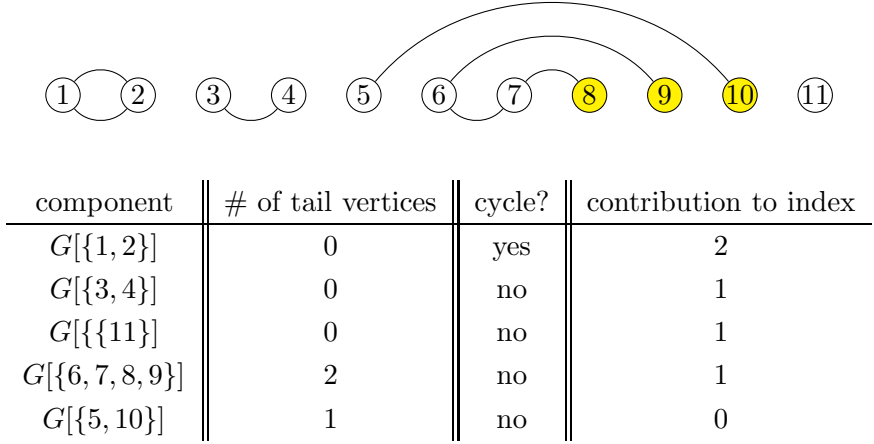
\end{example}

\begin{proof}

A cycle in $M_n^\C(\ul{a} \dd \ul{b})$ cannot contain vertices from $T$, and breaks into two cycles
in $\sigma_{n,\ul{a},\ul{b}}$. A path in $M_n^\C(\ul{a} \dd \ul{b})$ will be a cycle in $\sigma_{n,\ul{a},\ul{b}}$
containing the labels of all the vertices in the path. Thus (i) and (ii) are equivalent, so it suffices to prove (i).

By Corollary \ref{sums<n} and symmetry, it suffices to consider the case when $\sum a_i =n$ and $\sum b_i \leq n$.
Induct on $n$. The base case is trivial.

Given a symplectic meander $G$, let $f(G)$ denote the number of cycles plus the number of connected
components containing either 0 or 2 vertices from $T$ in $G$.

For the inductive step, first consider the case that $\ul{b}=\emptyset$, thus all vertices belong to the tail $T$.
There are no cycles in $M_n^\C(\ul{a} \dd \emptyset)$, and there are no connected components of  containing 0
vertices from $T$. Since each block of vertices $V_i$ is assigned $\lf a_i/2\rf$ top edges,
it follows that
\[f\left(M_n^\C(\ul{a} \dd \emptyset)\right)=\sum_{i=1}^{m}\lf\frac{a_i}{2}\rf = \ind \mf{p}_n^\C(\ul{a} \dd \emptyset),\]
by Theorem \ref{C parabolic}.

To complete the inductive step, now consider the case that $\ul{b}\neq\emptyset$. Suppose $a_1=b_1$.
Let $H$ denote the subgraph of $M_n^\C(\ul{a} \dd \ul{b})$ induced by the vertices labeled 1 through $a_1$,
and let $G$ denote the subgraph induced by the remaining vertices. Then $M_n^\C(\ul{a} \dd \ul{b})=H+G$
and $H$ contains no vertices from $T$. Clearly $f(H)=a_1$, and using the inductive hypothesis on $G$ we have
\[f\left(M_n^\C(\ul{a} \dd \ul{b})\right)=f(H)+f(G)
=a_1+\ind \mf{p}_{n-a_1}^\C((a_2,a_3,\dots a_m) \dd (b_2,b_3,\dots b_t)).\]
By Theorem \ref{symplectic index}, this is equal to $\ind \mf{p}_n^\C(\ul{a} \dd \ul{b})$.

Suppose $a_1\leq b_1/2$. By Theorem \ref{edge contraction} the meander
$\displaystyle G = M_{n-a_1}^\C((a_2,a_3,\dots a_m) \dd (b_1-2a_1,a_1,b_2,b_3,\dots b_t))$
can be obtained from $M_n^\C(\ul{a} \dd \ul{b})$ by edge contractions that
do not delete vertices from $T$. Thus by induction we have
\[f\left(M_n^\C(\ul{a} \dd \ul{b})\right)=f(G)
=\ind \mf{p}_{n-a_1}^\C((a_2,a_3,\dots a_m) \dd (b_1-2a_1,a_1,b_2,b_3,\dots b_t)).\]
And by Theorem \ref{symplectic index}, this is equal to $\ind \mf{p}_n^\C(\ul{a} \dd \ul{b})$.

Similarly, suppose $a_1>b_1/2$. By Theorem \ref{edge contraction} the meander
$\displaystyle G=M_{n-b_1+a_1}^\C((2a_1-b_1,a_2,a_3,\dots a_m) \dd (a_1,b_2,b_3,\dots b_t))$
can be obtained from $M_n^\C(\ul{a} \dd \ul{b})$ by edge contractions that
do not delete vertices from $T$. Thus by induction we have
\[f\left(M_n^\C(\ul{a} \dd \ul{b})\right)=f(G)
=\ind \mf{p}_{n-b_1+a_1}^\C((2a_1-b_1,a_2,a_3,\dots a_m) \dd (a_1,b_2,b_3,\dots b_t)).\]
Again by Theorem \ref{symplectic index}, this is equal to $\ind \mf{p}_n^\C(\ul{a} \dd \ul{b})$.

\end{proof}



The following corollary gives a necessary condition for a symplectic seaweed to have minimal index. This is analogous to Corollary \ref{2 odd parts}.

\begin{cor}\label{c necessary}
Without loss of generality, assume $\sum b_i\leq \sum a_i$. If $\ind \mf{p}_n^\C(\ul{a} \dd \ul{b})=0$, then
$\sum a_i=n$, and $\sum b_i=n-r<n$, and there must be exactly $r$ odd integers among $\ul{a}$ and $\ul{b}$.

\end{cor}

\begin{example}
The seaweed $\mf{p}_{15}^\C((10,5) \dd (5,8))$ satisfies the conclusion of Corollary \ref{c necessary}.
In fact, it is a Frobenius seaweed since its meander consists of two paths, each path containing exactly one
vertex from the tail, as shown below.
\begin{figure}[H]
\[\begin{tikzpicture}
\vertex (1) at (.9,0) {1};
\vertex (2) at (1.8,0) {2};
\vertex (3) at (2.7,0) {3};
\vertex (4) at (3.6,0) {4};
\vertex (5) at (4.5,0) {5};
\vertex (6) at (5.4,0) {6};
\vertex (7) at (6.3,0) {7};
\vertex (8) at (7.2,0) {8};
\vertex (9) at (8.1,0) {9};
\vertex (10) at (9,0) {10};
\vertex (11) at (9.9,0) {11};
\vertex (12) at (10.8,0) {12};
\vertex (13) at (11.7,0) {13};
\vertex[fill=yellow] (14) at (12.6,0) {14};
\vertex[fill=yellow] (15) at (13.5,0) {15};

\path
(1) edge[bend left=50,color=blue,line width=1.2pt] (10)
(2) edge[bend left=50,color=blue,line width=1.2pt] (9)
(3) edge[bend left=50,color=green,line width=1.2pt] (8)
(4) edge[bend left=50,color=blue,line width=1.2pt] (7)
(5) edge[bend left=50,color=blue,line width=1.2pt] (6)
(11) edge[bend left=50,color=green,line width=1.2pt] (15)
(12) edge[bend left=50,color=blue,line width=1.2pt] (14)

(1) edge[bend right=50,color=blue,line width=1.2pt] (5)
(2) edge[bend right=50,color=blue,line width=1.2pt] (4)
(6) edge[bend right=50,color=blue,line width=1.2pt] (13)
(7) edge[bend right=50,color=blue,line width=1.2pt] (12)
(8) edge[bend right=50,color=green,line width=1.2pt] (11)
(9) edge[bend right=50,color=blue,line width=1.2pt] (10)
;\end{tikzpicture}\]
\caption{$M_{15}^\C((10,5) \dd (5,8))$ with components highlighted}
\end{figure}
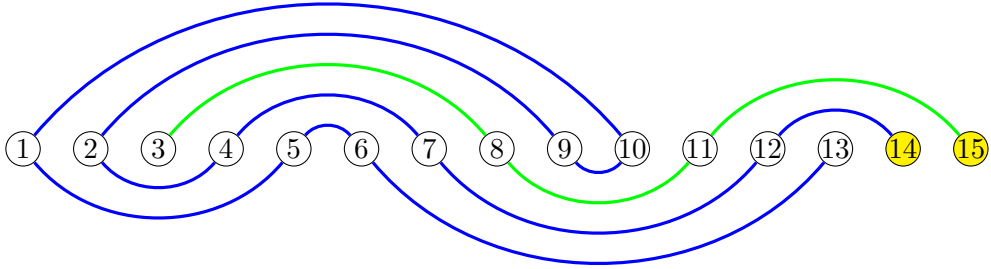

The seaweed $\mf{p}_{15}^\C((10,5) \dd (3,10))$ also satisfies the conclusion of Corollary \ref{c necessary}.
However it is not a Frobenius seaweed since its meander consists of two paths, one of which has two vertices
from the tail, and the other has zero vertices from the tail, as shown below.
In particular, $\ind \mf{p}_{15}^\C((10,5) \dd (3,10))=2$.
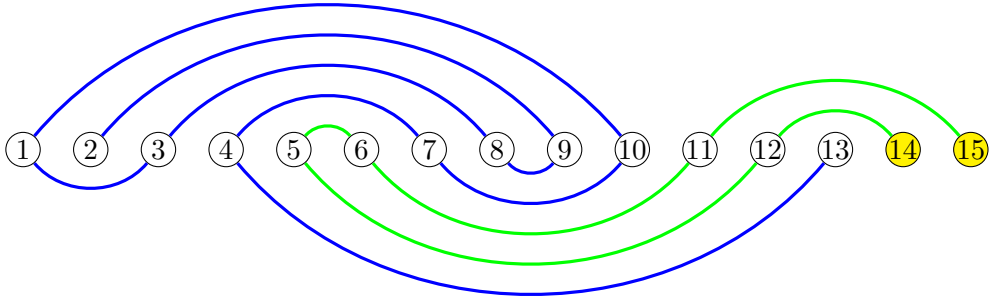
\begin{figure}[H]
\[\begin{tikzpicture}
\vertex (1) at (.9,0) {1};
\vertex (2) at (1.8,0) {2};
\vertex (3) at (2.7,0) {3};
\vertex (4) at (3.6,0) {4};
\vertex (5) at (4.5,0) {5};
\vertex (6) at (5.4,0) {6};
\vertex (7) at (6.3,0) {7};
\vertex (8) at (7.2,0) {8};
\vertex (9) at (8.1,0) {9};
\vertex (10) at (9,0) {10};
\vertex (11) at (9.9,0) {11};
\vertex (12) at (10.8,0) {12};
\vertex (13) at (11.7,0) {13};
\vertex[fill=yellow] (14) at (12.6,0) {14};
\vertex[fill=yellow] (15) at (13.5,0) {15};

\path
(1) edge[bend left=50,color=blue,line width=1.2pt] (10)
(2) edge[bend left=50,color=blue,line width=1.2pt] (9)
(3) edge[bend left=50,color=blue,line width=1.2pt] (8)
(4) edge[bend left=50,color=blue,line width=1.2pt] (7)
(5) edge[bend left=50,color=green,line width=1.2pt] (6)
(11) edge[bend left=50,color=green,line width=1.2pt] (15)
(12) edge[bend left=50,color=green,line width=1.2pt] (14)

(1) edge[bend right=50,color=blue,line width=1.2pt] (3)
(4) edge[bend right=50,color=blue,line width=1.2pt] (13)
(5) edge[bend right=50,color=green,line width=1.2pt] (12)
(6) edge[bend right=50,color=green,line width=1.2pt] (11)
(7) edge[bend right=50,color=blue,line width=1.2pt] (10)
(8) edge[bend right=50,color=blue,line width=1.2pt] (9)
;\end{tikzpicture}\]
\caption{$M_{15}^\C((10,5) \dd (3,10))$ with components highlighted}
\end{figure}
\end{example}

\section{Formulas}

Next we consider symplectic seaweed subalgebras where $\ul{a}$ and $\ul{b}$ have a small number of parts.
Our goal is to give a simple closed formula for the index, or at least characterize the seaweeds of index zero.
Theorem \ref{C parabolic} directly covers all cases when either $\ul{a}=\emptyset$ or $\ul{b}=\emptyset$. The
next case we consider is when $\ul{a}$ and $\ul{b}$ each have one part. This case is easily handled
by applying Theorem \ref{C inductive} and Theorem \ref{C parabolic} of Panyushev, and should be considered
a corollary of these results.

\begin{cor}
If $a=b$ then $\ind \mf{p}_n^\C((a) \dd (b))=n$.
Otherwise, without loss of generality assume that $a>b$, and we have
\[\ind \mf{p}_n^\C((a) \dd (b))=
\begin{cases}
n-a+\lf\frac{n-b}{2}\rf & \text{ if }n\text{ is even}\\
n-a+\lf\frac{n-b-1}{2}\rf & \text{ if }n\text{ is odd}.
\end{cases}\]

\end{cor}

\begin{proof}

If $n>a$ then $\ind \mf{p}_n^\C((a) \dd (b))=n-a+\ind \mf{p}_{n-a}^\C((a) \dd (b))$, so it suffices to show that
\begin{equation}\label{2 parts eq}
\ind \mf{p}_n^\C((n) \dd (b))=
\begin{cases}
\lf\frac{n-b}{2}\rf & \text{ if }n\text{ is even}\\
\lf\frac{n-b-1}{2}\rf & \text{ if }n\text{ is odd}.
\end{cases}
\end{equation}

Suppose $b\leq n/2$. By Theorem \ref{C inductive} and Theorem \ref{C parabolic} we have
\[\ind \mf{p}_n^\C((n) \dd (b))=\ind \mf{p}_n^\C((b) \dd (n))=\ind \mf{p}_{n-b}^\C((\emptyset) \dd (n-2b,b))
=\lf\frac{n-2b}{2}\rf+\lf\frac{b}{2}\rf.\]
If $n$ is even then $n-2b$ is even and
\[\lf\frac{n-2b}{2}\rf+\lf\frac{b}{2}\rf=\frac{n-2b}{2}+\lf\frac{b}{2}\rf=\lf\frac{n-b}{2}\rf.\]
If $n$ is odd then $n-2b$ is odd and
\[\lf\frac{n-2b}{2}\rf+\lf\frac{b}{2}\rf=\frac{n-2b-1}{2}+\lf\frac{b}{2}\rf=\lf\frac{n-b-1}{2}\rf.\]

Suppose $b>n/2$. We prove \eqref{2 parts eq} by induction on $n$. The base case is trivial. For the inductive step,
use Theorem \ref{C inductive}:
\[\ind \mf{p}_n^\C((n) \dd (b))=\ind \mf{p}_n^\C((b) \dd (n))=\ind \mf{p}_{b}^\C((2b-n) \dd (b)).\]
Using the inductive hypothesis, it is easy to show that \eqref{2 parts eq} holds by considering all four cases
for the parity of $n$ and $b$.

\end{proof}

The next (and last) case we consider is when $\ul{a}$ and $\ul{b}$ have a total of three parts.

In this case we
rely on the graph theoretic characterization of the index, given in Theorem \ref{symplectic index}.

\begin{thm}\label{3 parts thm1}
Let $a+b=n$. If $c=n-1$ or $c=n-2$ then

\begin{eqnarray}\label{Formula}
\ind \mf{p}_n^\C((a,b) \dd (c))=\gcd(a+b,b+c)-1.
\end{eqnarray}

\end{thm}

\begin{proof}

Suppose $c=n-1$. The underlying graphs of $M_n^\C((a,b) \dd (c))$ and $M^\A((a,b) \dd (c,1))$ are
isomorphic. The only difference is that for the symplectic meander, we call the single vertex labeled $n$ the tail.
By Theorem  \ref{A 4 parts}, the number of cycles
plus the number of connected components of this graph is $\gcd(a+b,b+c)$. One connected component
of $M_n^\C((a,b) \dd (c))$ contains 1 vertex from the tail, and the remaining
connected components contain 0 vertices from the tail. By Theorem \ref{symplectic index}, equation (\ref{Formula}) follows.

Suppose that $c=n-2$. The vertices labeled $n-1$ and $n$ are the tail $T$ of $M_n^\C((a,b) \dd (c))$.
By Theorem \ref{A 4 parts}, the number of cycles
plus the number of connected components of $M^\A((a,b) \dd (c,2))$ is $\gcd(a+b,b+c)$.
Suppose further that both vertices of $T$ belong to a path in $M^\A((a,b) \dd (c,2))$. By removing the bottom edge
from $n-1$ to $n$ in $M^\A((a,b) \dd (c,2))$, we see that $n-1$ and $n$ belong to separate connected components
in $M_n^\C((a,b) \dd (c))$, and each component has only 1 vertex from the tail. Equation \ref{Formula} follows.
On the other hand, suppose both vertices of $T$ belong to a cycle in $M^\A((a,b) \dd (c,2))$. By removing the bottom edge
from $n-1$ to $n$ in $M^\A((a,b) \dd (c,2))$, we see that $n-1$ and $n$ belong to the same connected component
in $M_n^\C((a,b) \dd (c))$, and that component is a path. Once again, equation \ref{Formula} follows.
\end{proof}

\begin{thm}\label{3 parts thm2}
If $a+b=n$, then $\ind \mf{p}_n^\C((a,b) \dd (c))=0$ if and only if one of the following conditions hold:
\begin{enumerate}[(i)]
\item $c=n-1$ and $\gcd(a+b,b+c)=1,$
\item $c=n-2$ and $\gcd(a+b,b+c)=1,$
\item $c=n-3$, the integers $a,b,$ and $c$ are all odd, and $\gcd(a+b,b+c)=2$.
\end{enumerate}
\end{thm}

\begin{example} Consider the seaweed $\mf{p}_{16}^\C((7,9) \dd (13))$. By Theorem \ref{3 parts thm2},
it is Frobenius. Below is its meander, with components highlighted.  Also note that one would need to apply seven signature moves (FRPFRPP) to see that its index is the same as the parabolic seaweed $\mf{p}_{3}^\C((1,1,1)\dd \emptyset$)




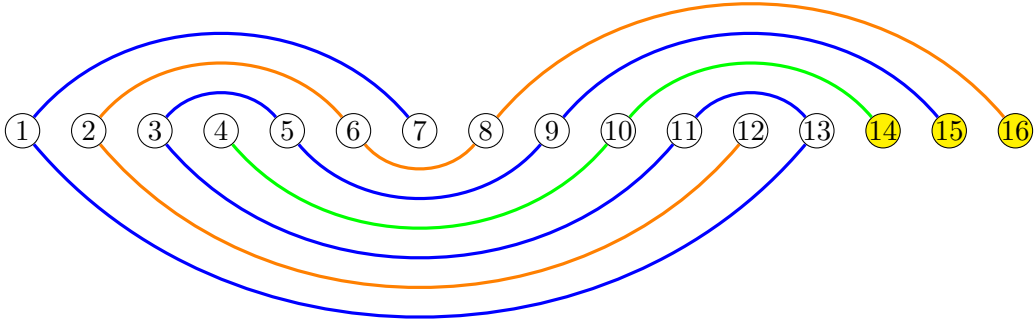
\begin{figure}[H]
\[\begin{tikzpicture}
\vertex (1) at (.88,0) {1};
\vertex (2) at (1.76,0) {2};
\vertex (3) at (2.64,0) {3};
\vertex (4) at (3.52,0) {4};
\vertex (5) at (4.4,0) {5};
\vertex (6) at (5.28,0) {6};
\vertex (7) at (6.16,0) {7};
\vertex (8) at (7.04,0) {8};
\vertex (9) at (7.92,0) {9};
\vertex (10) at (8.8,0) {10};
\vertex (11) at (9.68,0) {11};
\vertex (12) at (10.56,0) {12};
\vertex (13) at (11.44,0) {13};
\vertex[fill=yellow] (14) at (12.32,0) {14};
\vertex[fill=yellow] (15) at (13.2,0) {15};
\vertex[fill=yellow] (16) at (14.08,0) {16};

\path
(1) edge[bend left=50,color=blue,line width=1.2pt] (7)
(2) edge[bend left=50,color=orange,line width=1.2pt] (6)
(3) edge[bend left=50,color=blue,line width=1.2pt] (5)
(8) edge[bend left=50,color=orange,line width=1.2pt] (16)
(9) edge[bend left=50,color=blue,line width=1.2pt] (15)
(10) edge[bend left=50,color=green,line width=1.2pt] (14)
(11) edge[bend left=50,color=blue,line width=1.2pt] (13)

(1) edge[bend right=50,color=blue,line width=1.2pt] (13)
(2) edge[bend right=50,color=orange,line width=1.2pt] (12)
(3) edge[bend right=50,color=blue,line width=1.2pt] (11)
(4) edge[bend right=50,color=green,line width=1.2pt] (10)
(5) edge[bend right=50,color=blue,line width=1.2pt] (9)
(6) edge[bend right=50,color=orange,line width=1.2pt] (8)
;\end{tikzpicture}\]
\caption{$M_{16}^\C((7,9)\dd (13))$ with components highlighted}
\end{figure}
\end{example}

\begin{proof}

Suppose $c\leq n-4$. Then by Corollary \ref{c necessary} $\ind \mf{p}_n^\C((a,b) \dd (c))\neq 0$
since we have at most three odd integers among $a,b,$ and $c$. The cases $c=n-1$ and $c=n-2$ follow
immediately from Theorem \ref{3 parts thm1}.

The remaining case $c=n-3$ requires more work. First suppose that $\ind \mf{p}_n^\C((a,b) \dd (c))=0$.
By Corollary \ref{c necessary}, $a,b,$ and $c$ must all be odd integers. By Theorem \ref{symplectic index},
the meander $M_n^\C((a,b) \dd (c))$ consists of three paths, and each path contains exactly one vertex
from the tail $T=\{n-2,n-1,n\}$. If we let $e$ denote an edge from $n-2$ to $n$, then the underlying graph
of $M_n^\C((a,b) \dd (c))+e$ is isomorphic to $M_n^\A((a,b) \dd (c,3))$. Since $M_n^\A((a,b) \dd (c,3))$
consists of two paths, by Theorem \ref{A 4 parts} we have $\gcd(a+b,b+c)=2$.

Finally, suppose that $c=n-3$, the integers $a,b,c$ are all odd, and $\gcd(a+b,b+c)=2$. Again we use the fact
that the underlying graph $M_n^\C((a,b) \dd (c))+e$ is isomorphic to $M_n^\A((a,b) \dd (c,3))$. Note that $a,b,c,3$ are
all odd and $\gcd(a+b,b+c)=2$. Using Theorem \ref{A 4 parts}, we conclude that $M_n^\A((a,b) \dd (c,3))$ has no
cycles, and consists of two paths. Therefore, $M_n^\C((a,b) \dd (c))$ consists of three paths, and in particular
the vertices $n-2$ and $n$ belong to distinct components.

Next we show that the vertices $n-1$ and $n$
belong to distinct components in $M_n^\C((a,b) \dd (c))$. If we follow the pair of edges starting from
$n-1$ and $n$, they will connect with another pair of consecutive vertices (i.e. vertices whose labels are consecutive
integers). We successively follow pairs of edges, until either
\begin{enumerate}
\item one (or both) of the edges connects to a vertex which is the end of a path,
\item there is a single edge connecting the consecutive vertices,
\item the pair of edges connects to vertices that are not consecutive.
\end{enumerate}
If (1) occurs then we are done. Because $a,b,c$ are all odd, (2) cannot occur. The only way (3)
can occur is if the pair edges are top edges originating from vertices $a$ and $a+1$, and terminate at 1 and $n$ respectively.
But the path containing the vertex $n-1$ must be the path containing the vertex $a$. This is because the paths
begin with the component containing $n-1$ on the left, and the component containing $n$ on the right, and this switches
every time we follow one pair of paths. So after following a pair of bottom edges to reach vertices $a$ and $a+1$, the
component containing $n-1$ is on the right, and the component containing $n$ is on the left. But then the path originating
at $n$ also terminates at $n$, so (3) is impossible as well.

Finally, we use a similar technique to show that the vertices $n-2$ and $n-1$ belong to distinct components of
$M_n^\C((a,b) \dd (c))$. Follow pairs of edges starting from $n-2$ and $n-1$. As above (2) is impossible.
If (3) occurs, then the pair of edges are top edges originating from vertices $a$ and $a+1$, and terminate at 1 and $n$ respectively.
But the path containing $n-1$ is on the left, so it continues from vertex $a+1$ to $n$. However, we have already shown that
$n-1$ and $n$ must belong to distinct components. Therefore (3) is impossible, and (1) must occur as desired.

\end{proof}

\begin{thm}\label{3 parts thm3}  Let $a$, $b$, and $n$ be positive integers.

\begin{enumerate}[(i)]

\item If $a+b=n-1$ then $\ind \mf{p}_n^\C((n) \dd (a,b))=\gcd(a+b,b+1)-1$.

\item If $a+b=n-2$ then $\ind \mf{p}_n^\C((n) \dd (a,b))=\gcd(a+b,b+2)-1$.

\end{enumerate}

\end{thm}

\begin{proof}

The proof is nearly identical to that of Theorem \ref{3 parts thm1}, and is omitted.


\end{proof}

\begin{thm}\label{3 parts thm4}
The index of $\mf{p}_n^\C((n) \dd (a,b))$ is equal to zero
if and only if one of the following conditions hold
\begin{enumerate}[(i)]
\item $a+b=n-1$ and $\gcd(a+b,b+1)=1$

\item $a+b=n-2$ and $\gcd(a+b,b+2)=1$

\item $a+b=n-3$, the integers $n,a,$ and $b$ are all odd, and $\gcd(a+b,b+3)=2$.

\end{enumerate}
\end{thm}

\begin{proof}

As was the case with Theorem \ref{3 parts thm2}, if $c\leq 4$ then by Corollary \ref{c necessary}
it follows that
$\ind \mf{p}_n^\C((n) \dd (a,b))\neq 0$
since we have at most three odd integers among $a,b,$ and $n$. And the cases $a+b=n-1$ and $a+b=n-2$ follow
immediately from Theorem \ref{3 parts thm3}.

Assume $a+b=n-3$. We prove necessity in a manner similar to the proof of Theorem \ref{3 parts thm2}.
Suppose that $\ind \mf{p}_n^\C((n) \dd (a,b))=0$.
By Corollary \ref{c necessary}, $a,b,$ and $n$ must all be odd integers. By Theorem \ref{symplectic index},
the meander $M_n^\C((n) \dd (a,b))$ consists of three paths, and each path contains exactly one vertex
from the tail $T=\{n-2,n-1,n\}$. If we let $e$ denote an edge from $n-2$ to $n$, then the underlying graph
of $M_n^\C((n) \dd (a+b))+e$ is isomorphic to $M_n^\A((n) \dd (a,b,3))$. Since $M_n^\A((n) \dd (a,b,3))$
consists of two paths, by Theorem \ref{A 4 parts} we have $\gcd(a+b,b+3)=2$.

Although we could prove sufficiency in a manner analogous to the proof of Theorem \ref{3 parts thm2},
we take a different approach here. Suppose that $a,b,$ and $n$ are odd integers,
and $\gcd(a+b,b+3)=2$.
If $n=5$, there is only one such seaweed to consider, and it is
easy to check that it satisfies condition (iii) of Theorem \ref{3 parts thm4}. Now assume $n\geq 7$.
We perform a series of edge contractions on the meander $M_n^\C((n) \dd (a,b))$,
keeping track of the vertices we identify as the tail $T$. Initially $T=\{n-2, n-1, n\}$, and we contract the
top edges connected to these vertices, and identify the tail as $T=\{1,2,3\}$. After reflecting horizontally and vertically,
we see that
\[\ind \mf{p}_n^\C((n) \dd (a,b))=\ind \mf{p}_{n-3}^\C((b,a) \dd (n-6)).\]
Since $\gcd(a+b,b+3)=2$,
there exists integers $k$ and $m$ such that
\[k(a+b)+m(b+3)=2.\]
Furthermore,
\[(k+2m)(a+b)-m(a+n-6)=(k+2m)(a+b)-m(2a+b-3)=k(a+b)+m(b+3)=2.\]
Since $a+b$ and $a+n-6$ are even, this implies that $\gcd(b+a,a+n-6)=2$. So by Theorem \ref{3 parts thm2}
we have
\[\ind \mf{p}_n^\C((n) \dd (a,b))=\ind \mf{p}_{n-3}^\C((b,a) \dd (n-6))=0.\]
\end{proof}

\noindent
\textit{Remark:}  Consider the symplectic meander
$M^C_n ((a,b,c)\dd (d))$ where $a+b+c=n$ and $d<n$.  The index computations for this meander are analogous to those for the meander $M^A_n ((a,b,c)\dd (n-d,d))$ which by Theorem \ref{5 parts} has no linear gcd formula for its index.  We therefore do not expect that there is a closed linear gcd formula for the index of a symplectic meander of this form.

\bigskip
\noindent
\textit{Acknowledgments:}
The authors are grateful to Murray Gerstenhaber, Tony Giaquinto,  and Jim Stasheff for a number of helpful discussions.

\end{document}